\documentclass[12pt]{amsart}
\usepackage{amsmath,amssymb,amsbsy,amsfonts,latexsym,amsopn,amstext,cite,
                                               amsxtra,euscript,amscd,bm,mathabx}
\usepackage{url}
\usepackage{amsbsy} 
\usepackage[colorlinks,linkcolor=blue,anchorcolor=blue,citecolor=blue,backref=page]{hyperref}
\usepackage{color}
\usepackage{graphics,epsfig}
\usepackage{graphicx}
\usepackage{float} 
\usepackage[english]{babel} 
\usepackage{mathtools}
\usepackage{todonotes}
\usepackage{url}
\usepackage[colorlinks,linkcolor=blue,anchorcolor=blue,citecolor=blue,backref=page]{hyperref}

\def \newr{r}

\usepackage[norefs,nocites]{refcheck}

\hypersetup{breaklinks=true}

\usepackage[norefs,nocites]{refcheck}
\usepackage[english]{babel}
\begin{document}

\newtheorem{thm}{Theorem}
\newtheorem{lem}[thm]{Lemma}
\newtheorem{claim}[thm]{Claim}
\newtheorem{cor}[thm]{Corollary}
\newtheorem{prop}[thm]{Proposition} 
\newtheorem{definition}[thm]{Definition}
\newtheorem{rem}[thm]{Remark} 
\newtheorem{question}[thm]{Question}
\newtheorem{qn}[thm]{Question}
\newtheorem{conj}[thm]{Conjecture}
\newtheorem{prob}{Problem}
\newtheorem{assumption}[thm]{Assumption}

\newcommand{\GL}{\operatorname{GL}}
\newcommand{\SL}{\operatorname{SL}}
\newcommand{\lcm}{\operatorname{lcm}}
\newcommand{\ord}{\operatorname{ord}}
\newcommand{\Op}{\operatorname{Op}}
\newcommand{\Tr}{\operatorname{Tr}}
\newcommand{\Nm}{\operatorname{Nm}}

\numberwithin{equation}{section}
\numberwithin{thm}{section}
\numberwithin{table}{section}

\def\vol {{\mathrm{vol\,}}}
\def\squareforqed{\hbox{\rlap{$\sqcap$}$\sqcup$}}
\def\qed{\ifmmode\squareforqed\else{\unskip\nobreak\hfil
\penalty50\hskip1em\null\nobreak\hfil\squareforqed
\parfillskip=0pt\finalhyphendemerits=0\endgraf}\fi}

\def \balpha{\bm{\alpha}}
\def \bbeta{\bm{\beta}}
\def \bgamma{\bm{\gamma}}
\def \blambda{\bm{\lambda}}
\def \bchi{\bm{\chi}}
\def \bphi{\bm{\varphi}}
\def \bpsi{\bm{\psi}}
\def \bomega{\bm{\omega}}
\def \btheta{\bm{\vartheta}}
\def \ochi{\overline{\chi}}
\def \h{\widehat h}

\def\ovG {\overline{\Gamma}}

\def\uu {\mathbf{u}} 
\def\uv {\mathbf{v}} 
\def\uk {\mathbf{k}} 
\def\ut {\mathbf{t}}

\def\eps{\varepsilon}

\newcommand{\bfxi}{{\boldsymbol{\xi}}}
\newcommand{\bfrho}{{\boldsymbol{\rho}}}

\def\Kab{\sfK_\psi(a,b)}
\def\Kuv{\sfK_\psi(u,v)}
\def\SaUV{\cS_\psi(\balpha;\cU,\cV)}
\def\SaAV{\cS_\psi(\balpha;\cA,\cV)}

\def\SUV{\cS_\psi(\cU,\cV)}
\def\SAB{\cS_\psi(\cA,\cB)}

\def\Kmnp{\sfK_p(m,n)}

\def\KKap{\cH_p(a)}
\def\KKaq{\cH_q(a)}
\def\KKmnp{\cH_p(m,n)}
\def\KKmnq{\cH_q(m,n)}

\def\Klmnp{\sfK_p(\ell, m,n)}
\def\Klmnq{\sfK_q(\ell, m,n)}

\def \SALMNq {\cS_q(\balpha;\cL,\cI,\cJ)}
\def \SALMNp {\cS_p(\balpha;\cL,\cI,\cJ)}

\def \SACXMQX {\fS(\balpha,\bzeta, \bxi; M,Q,X)}

\def\SAMJp{\cS_p(\balpha;\cM,\cJ)}
\def\SAMJq{\cS_q(\balpha;\cM,\cJ)}
\def\SAqMJq{\cS_q(\balpha_q;\cM,\cJ)}
\def\SAJq{\cS_q(\balpha;\cJ)}
\def\SAqJq{\cS_q(\balpha_q;\cJ)}
\def\SAIJp{\cS_p(\balpha;\cI,\cJ)}
\def\SAIJq{\cS_q(\balpha;\cI,\cJ)}

\def\RIJp{\cR_p(\cI,\cJ)}
\def\RIJq{\cR_q(\cI,\cJ)}

\def\TWXJp{\cT_p(\bomega;\cX,\cJ)}
\def\TWXJq{\cT_q(\bomega;\cX,\cJ)}
\def\TWpXJp{\cT_p(\bomega_p;\cX,\cJ)}
\def\TWqXJq{\cT_q(\bomega_q;\cX,\cJ)}
\def\TWJq{\cT_q(\bomega;\cJ)}
\def\TWqJq{\cT_q(\bomega_q;\cJ)}

 \def \xbar{\overline x}
  \def \ybar{\overline y}

\def\cA{{\mathcal A}}
\def\cB{{\mathcal B}}
\def\cC{{\mathcal C}}
\def\cD{{\mathcal D}}
\def\cE{{\mathcal E}}
\def\cF{{\mathcal F}}
\def\cG{{\mathcal G}}
\def\cH{{\mathcal H}}
\def\cI{{\mathcal I}}
\def\cJ{{\mathcal J}}
\def\cK{{\mathcal K}}
\def\cL{{\mathcal L}}
\def\cM{{\mathcal M}}
\def\cN{{\mathcal N}}
\def\cO{{\mathcal O}}
\def\cP{{\mathcal P}}
\def\cQ{{\mathcal Q}}
\def\cR{{\mathcal R}}
\def\cS{{\mathcal S}}
\def\cT{{\mathcal T}}
\def\cU{{\mathcal U}}
\def\cV{{\mathcal V}}
\def\cW{{\mathcal W}}
\def\cX{{\mathcal X}}
\def\cY{{\mathcal Y}}
\def\cZ{{\mathcal Z}}
\def\Ker{{\mathrm{Ker}}}
\def\g{{\mathrm{gcd}}}
\def\rad{{\mathrm{rad}}}
\def\rank{{\mathrm{rank}}}
\def\Gal{{\mathrm{Gal}}}

\def\rn{{\mathrm{new}}}

\def\Psf {\mathsf P}
\def\Qsf {\mathsf Q}
\def\fsf {\mathsf f}
\def\ksf {\mathsf k}
\def\hsf {\mathsf h}
\def\isf{\mathsf O}

\def\Dsf {\mathsf D}

\def\NmQR{N(m;Q,R)}
\def\VmQR{\cV(m;Q,R)}

\def\Xm{\cX_m}

\def \A {{ \mathbb A}}
\def \B {{ \mathbb B}}
\def \C {{ \mathbb C}}
\def \N {{ \mathbb N}}
\def \F {{ \mathbb F}}
\def \G {{\mathbb G}}
\def \L {{\mathbb L}}
\def \K {{\mathbb K}}
\def \PP {{\mathbb P}}
\def \Q {{\mathbb Q}}
\def \R {{\mathbb R}}
\def \Z {{\mathbb Z}}

\def \fS{\mathfrak S}
\def \fA{\mathfrak A}
\def \fB{\mathfrak B}
\def \fC{\mathfrak C}
\def \fD{\mathfrak D}

\def\e{{\mathbf{\,e}}}
\def\ep{{\mathbf{\,e}}_p}
\def\eq{{\mathbf{\,e}}_q}
\def\er{{\mathbf{\,e}}_R}
\def\\{\cr}
\def\({\left(}
\def\){\right)}
\def\fl#1{\left\lfloor#1\right\rfloor}
\def\rf#1{\left\lceil#1\right\rceil}

\def\Tr{{\mathrm{Tr}}}
\def\Nm{{\mathrm{Nm}}}
\def\Im{{\mathrm{Im}}}

\def \oF {\overline \F}

\newcommand{\pfrac}[2]{{\left(\frac{#1}{#2}\right)}}

\def \Prob{{\mathrm {}}}
\def\e{\mathbf{e}}
\def\ep{{\mathbf{\,e}}_p}
\def\epp{{\mathbf{\,e}}_{p^2}}
\def\em{{\mathbf{\,e}}_m}

\def\Res{\mathrm{Res}}
\def\Orb{\mathrm{Orb}}

\def\vec#1{\mathbf{#1}}
\def \va{\vec{a}}
\def \vb{\vec{b}}
\def \vb{\vec{c}}
\def \vm{\vec{m}}
\def \vu{\vec{u}}
\def \vv{\vec{v}}
\def \vx{\vec{x}}
\def \vy{\vec{y}}
\def \vz{\vec{z}}

\newcommand{\gvec}{\boldsymbol{\gamma}}

\def\flp#1{{\left\langle#1\right\rangle}_p}
\def\T {\mathsf {T}}

\def\sfG {\mathsf {G}}
\def\sfK {\mathsf {K}}

\def\mand{\qquad\mbox{and}\qquad}

\title[Equations with elements of finitely generated groups]  
{Systems of equations with elements of finitely generated groups  and with 
linear recurrence sequences}

\author[S. Bhakta]{Subham Bhakta}
\address{School of Mathematics and Statistics, University of New South Wales, Sydney, NSW 2052, Australia.} 
\email{subham.bhakta@unsw.edu.au}

\author[A. Ostafe]{Alina Ostafe}
\address{School of Mathematics and Statistics, University of New South Wales, Sydney, NSW 2052, Australia.} 
\email{alina.ostafe@unsw.edu.au}

\author[I. E. Shparlinski]{Igor E. Shparlinski}
\address{School of Mathematics and Statistics, University of New South Wales, Sydney, NSW 2052, Australia.} 
\email{igor.shparlinski@unsw.edu.au}

\begin{abstract} We obtain new and essentially optimal bounds on the number of solutions 
of diagonal systems of $k$ polynomial equations in $s$ variables from an arbitrary finite subset of a finitely generated group  
 in a field of characteristic zero.  Our bounds generalise and sometimes improve previous results 
 of J.~Bourgain, M.~Z.~Garaev, S.~V.~Konyagin, and I.~E.~Shparlinski (2014), for $k=1$, and 
 of A.~Manning, A.~Ostafe  and I.~E.~Shparlinski (2026), for $k=2$. In turn, this improvement 
 leads to sharpening a result of A.~Manning, A.~Ostafe  and I.~E.~Shparlinski (2026) concerning 
 arithmetic statistics of certain matrices. 
 
We also obtain similar results for  $s$-tuples of linear recurrence sequences, 
 in the case of linear equations and quadratic equations.
\end{abstract}  


\maketitle
\tableofcontents
\section{Introduction}
\subsection{Set-up}
Let $K$ be an algebraically closed field of characteristic zero. 

Our main goal is to estimate the number of solutions to system of $k$ diagonal equations in $s$ variables, see~\eqref{eq:GenEqnsSubgr}, in elements of a finite subset $\cW$  of a finitely generated subgroup $\Gamma \subseteq K^{\ast}$. 
Our bounds depend only on the cardinality $W = \sharp\, \cW$ (also on generators of $\Gamma$ and the parameters $k$ and $s$) but do not depend on the particular choice of elements of $\cW$ or the 
coefficients of equations in~\eqref{eq:GenEqnsSubgr}.  Such equations have been considered in~\cite{BGKS} for $k=1$ and in~\cite{MOS} for $k=2$. Here we generalise and, for $k=2$,  
also improve these results. In fact, in many natural cases our results are optimal. 

We also obtain a {\it paucity\/} result for symmetric systems of equations~\eqref{eq:VMVT}. 
We note that our result has a flavour of so-called {\it diameter-free\/} estimate, which has been studied 
in the context of the Vinogrdov Mean Value Theorem, see~\cite{Mud, Wool2}. 

Certainly one of our main tools comes from bounds on the number of degenerate solutions to a single equation in elements of $\Gamma$. We combine it with new combinatorial arguments to count, 
in an optimal way,  the total number of solutions to a system of diagonal equations.

We also use similar ideas to study another apparently new question concerning equations with 
linear recurrence sequences. 

We recall that a linear recurrence sequence  of order $d\ge 1$ is a sequence $\vu=\(u(x)\)_{x \ge 1}$ of elements in $K$, which satisfies a \textit{minimal} relation
\begin{equation}\label{eq:rel}
u(x+d)=c_{d-1}u(x+d-1)+\ldots+c_0u(x),\quad x \ge 1,
\end{equation}
with $c_0,\ldots,c_{d-1}\in K$, $c_0 \ne 0$, and some initial values $u(1),\dots,u(d) \in K$. 

We say that $\vu$ is simple if its  characteristic  polynomial
\[
f(T)=T^d-c_{d-1}T^{d-1}-\ldots-c_0 \in K[T],
\]
has only simple roots  $\alpha_1,\ldots,\alpha_d\in K$. We call $\alpha_1,\ldots,\alpha_d$ the characteristic roots of $\vu$. In this case, we can write
\begin{equation}
\label{eq:LRS char roots}
u(x)=\beta_{1} \alpha_1^{x}+\ldots+\beta_{d} \alpha_d^{x},\quad x\ge 1,
\end{equation}
for some $\beta_1,\ldots,\beta_d \in K$ depending on the initial values $u(1),\dots,u(d)$ and $\alpha_1,\ldots, \alpha_d$. Since the relation in~\eqref{eq:rel} is in minimal form, we have
\begin{equation}
\label{eq:char roots}
\alpha_1,\ldots,\alpha_d,\beta_1,\ldots,\beta_d \in K^\ast.
\end{equation}

We say that $\vu$ is non-degenerate, if neither any of the
$\alpha_1,\ldots,\alpha_d$, nor any ratio $\alpha_{i}/\alpha_{j}$ with $1\le i\ne j\le d$ is a root of unity. 
Throughout the whole paper, we assume that $\vu$ is simple and non-degenerate.

By the classical Skolem--Mahler--Lech theorem, any such $\vu$ contains only finitely many zeros. For the best known upper bound on the number of zeros, see~\cite[Theorem~1.2]{AV11}.
Here, we estimate the number of solutions to linear and non-linear equations with elements in a linear recurrence sequence. 

We would like to {\it emphasise\/} that all our bounds, for both finitely generated subgroups and
linear recurrence sequences, are fully uniform with respect to the coefficients of the 
equations. This is crucial for some applications, in particular, to some questions of arithmetic statistics 
of matrices; see~Section~\ref{sec:comments} and Appendix~\ref{app:CountMatr_CharPoly}. 

\subsection{Equations with elements of  finitely generated subgroups}\label{sec:main}
Let us recall that $\cW$ is a finite subset of a finitely generated subgroup $\Gamma \subseteq K^{\ast}$. 
First, we generalise and improve~\cite[Lemma~3.4]{ManOstShp}.

Let $k\ge 1$ be any integer and $\vec{A}=(\vec{a}_1,\ldots,\vec{a}_k)\in K^{(s+1)\times k}$. We denote by $M_{\vec{A}}(\cW)$ the number of solutions to the system of equations 
\begin{equation}\label{eq:GenEqnsSubgr}
\begin{split}
&a_{j,1}w_1^j +\ldots+a_{j,s}w_s^j = a_{j,0}, \qquad j =1, \ldots, k,\\
& \qquad  \qquad  \qquad w_1,\ldots, w_s  \in \cW.
\end{split}
\end{equation}

We say that the matrix of the coefficients $\vec{A}=(\vec{a}_1,\ldots,\vec{a}_k)\in K^{(s+1)\times k}$ is {\it non-degenerate\/} if 
for any $1\le I \le \min\{k,s\}$ and any subset
$\cI\subseteq \{1,\ldots,s\}$ with $\sharp\, \cI=I$, the system of $I$ equations in $I$ variables
\begin{equation}\label{eq:NonDegen A}
\sum_{i\in \mathcal I} a_{j,i}c_i^{j}=0,
\qquad j=1,\ldots,I,
\end{equation}
has no solution in $c_i\in K^\ast$, $i\in\mathcal I$.

\begin{thm}\label{thm:SystSubgrGen}
Let  $\cW \subseteq \Gamma$ be a finite subset of cardinality $W$ of a   subgroup of $\Gamma \subseteq K^{\ast}$ of fixed rank and let $\vec{A}=(\vec{a}_1,\ldots,\vec{a}_k)\in  K^{(s+1)\times k}$ be non-degenerate. Then we have 
\[
M_{\vec{A}}(\cW) \ll W^{\lfloor (s-m)/(k+1) \rfloor},
\]
where 
\[
m=  \begin{cases} 0,  & \text{if }~a_{1,0}=\ldots=a_{k,0}=0,\\
\min\left\{1\le j \le k:~a_{j,0}\ne 0\right\}, & \text{otherwise}. 
\end{cases} 
\]
\end{thm}

In Section~\ref{sec:low} we show that the bound of Theorem~\ref{thm:SystSubgrGen} is tight.

We now give a natural example of coefficients $\(a_{j,i}\)_{\substack{1\le j\le k\\ 1\le i \le s}}$ that satisfy the required conditions in Theorem~\ref{thm:SystSubgrGen}.

\begin{cor}\label{cor:PowerSumVals}
For any $\vec{a}=(a_1,\ldots, a_k)\in K^k$, we denote by $M^*_{\vec{a}}(\cW)$ the number of solutions to the system of equations 
    \begin{equation}\label{eq:GenEqnsSubgrAllOne}
\begin{split}
&w_1^j +\ldots+w_s^j = a_j, \qquad j =1, \ldots, k,\\
& \qquad  \qquad  w_1,\ldots, w_s  \in \cW.
\end{split}
\end{equation}
Then we have 
 \[
 M^*_{\vec{a}}(\cW) \ll W^{\lfloor (s-m)/(k+1) \rfloor},
\]
where 
\[
m=  \begin{cases} 0,  & \text{if }~a_{1}=\ldots=a_{k}=0,\\
\min\left\{1\le j \le k:~a_{j}\ne 0\right\}, & \text{otherwise}. 
\end{cases} 
\]
\end{cor}

Indeed, 
suppose that there exists a non-empty subset $\cI\subseteq \{1,\ldots,s\}$ with $ I=\sharp\, \cI \le  \min\{k,s\}$, such that the system of equations
\[
\sum_{i\in \mathcal I} c_i^{j}=0,
\qquad j=1,\ldots,I,
\]
has some solutions with $c_i\in K^\ast$, $i\in\mathcal I$. Then (rewriting the right-hand side above as a sum of powers of zeroes) by Lemma~\ref{lem:VMVTsmalls}, we have $c_i=0$ for any $i\in \cI$,  which is a contradiction. This shows that the required condition in Theorem~\ref{thm:SystSubgrGen} is satisfied. 

As an application of Theorem~\ref{thm:SystSubgrGen}, we give sharper versions of some results of~\cite{ManOstShp} on 
 counting some special matrices, see Appendix~\ref{app:CountMatr_CharPoly} for more details.

Furthermore, motivated by results of Wooley~\cite{Wool1,Wool3}, 
one can also ask about the so-called {\it paucity\/} of solutions to the 
system of equations
\begin{equation}\label{eq:VMVT}
\begin{split}
&w_1^j +\ldots+w_s^j = w_{s+1}^j +\ldots+w_{2s}^j,  \qquad j =1, \ldots k,\\
& \qquad  \qquad  w_1,\ldots, w_s, w_{s+1},\ldots, w_{2s}  \in \cW, 
\end{split}
\end{equation}
that is, solutions 
with the additional restriction 
\begin{equation}\label{eq:notpermutation}
w_1,\ldots, w_s \quad \text{is not a permutation of} \quad  w_{s+1},\ldots, w_{2s}.
\end{equation}

More precisely, let us denote by $J_{s,k}(\cW)$ the number of solutions to the 
system of equations~\eqref{eq:VMVT} with additional condition~\eqref{eq:notpermutation}.

We observe that without the condition~\eqref{eq:notpermutation}  the 
system of equations~\eqref{eq:VMVT}  has at least $W^s$ solutions. So we
now seek bounds of the form $J_{s,k}(\cW) = o(W^s)$.

First of all, note that Lemma~\ref{lem:VMVTsmalls} implies
\begin{equation}\label{eq:Jss}
J_{s,k}(\cW)=0\quad\text{for}\quad 1\le s\le k.
\end{equation}

Therefore, we need to estimate $J_{s,k}$ for $1\le k \le s-1$. It is clear that the coefficients in~\eqref{eq:VMVT} do not satisfy the required condition in Theorem~\ref{thm:SystSubgrGen}. We prove the following estimate.

\begin{thm}\label{thm:VMVT}
Let  $k,s$ be fixed integers with $1\le k\le s-1$, and let $\cW \subseteq \Gamma$ be a finite subset of cardinality $W$ of a subgroup of $\Gamma \subseteq K^{\ast}$ of fixed rank. Then we have
\[J_{s,k}(\cW) \ll W^{s-k+1}.\]
\end{thm}

As shown in Section~\ref{sec:low} the bound of Theorem~\ref{thm:VMVT} is also tight.

We also remark that Theorem~\ref{thm:VMVT} implies that  the total number of solutions to the 
system of equations~\eqref{eq:VMVT} (without any  additional conditions) is $M_s(W)  + O\(W^{s-k+1}\)$, 
for any $s \ge k+1$, where  
\[
M_s(W)  =  \sum_{j = 1}^s \binom{W}{j}  \sum_{\substack{m_1 + \ldots + m_j = s\\m_1, \ldots, m_j \ge 1}} \(\frac{s!}{m_1!\ldots m_j!}\)^2
\] (clearly, the terms with $j \le s-k+1$ 
can be absorbed in the error term  $O\(W^{s-k+1}\)$).

\subsection{Equations with elements of linear recurrence sequences}
Let $\vu=\(u(x)\)_{x \ge 1}$ be a simple and non-degenerate linear recurrence  of order $d\ge 1$ with characteristic roots $\alpha_1,\ldots,\alpha_d$, which are fixed in the rest of the paper. Let $s\ge 1$ be an integer and $\va=(a_0,a_1,\ldots, a_s)\in K\times (K^\ast)^{s}$ be an $(s+1)$-tuple. We denote by $L_{\va}(H)$ the number of integer solutions to the equation
\[
a_1u(x_1)+\ldots+a_su(x_s)=a_0, \qquad 1\le x_1,\ldots, x_s \le H.
\]

Clearly, we only need to consider the two cases 
\[
a_0 = 0 \mand a_0 = 1.
\]

Note that similar equations modulo primes have been studied in~\cite{BBG}, relying on the estimates for the exponential sums associated to the linear recurrence sequences over prime fields. Here to address the more difficult problem that we consider, we argue in the same spirit as in the proof of Theorem~\ref{thm:SystSubgrGen}. We use similar combinatorial arguments for bounding the number of degenerate solutions to the system of diagonal equations, determined by~\eqref{eq:LRS char roots}.

We prove the following in Section~\ref{sec:GenEq}, in fact we establish there a more general 
Theorem~\ref{thm:anycombination}.

\begin{thm}\label{thm:lrslineqn}
Let $\vu$ be simple and non-degenerate.  Then uniformly over $\va \in K\times (K^\ast)^{s}$, we have
 \[
    L_{\va}(H)\ll \begin{cases} H^{\lfloor s/2 \rfloor} & \text{if} \ a_0 = 0,\\
 H^{\lfloor (s-1)/2 \rfloor}& \text{if} \ a_0 = 1. 
 \end{cases} 
\]
\end{thm}

See Section~\ref{sec:comments} for possible applications of  Theorem~\ref{thm:lrslineqn}. 

Moreover, for $\vec{A} = \(\vec{a}_1, \vec{a}_2\)$, where 
\[
\vec{a}_\nu=(a_{\nu,0}, a_{\nu,1},\ldots, a_{\nu,s})\in K\times (K^\ast)^{s}, \qquad \nu =1,2, 
\]
let $N_{\vec{A}}(H)$ denote the number of solutions to
the system of equations
\begin{equation}\label{eq:geneqnsLRS}
\begin{split}
&a_{1,1}u(x_1)+\ldots+a_{1,s}u(x_s) = a_{1,0},\\
&a_{2,1}u(x_1)^2+\ldots+a_{2,s}u(x_s)^2 = a_{2,0}, \\
& \qquad \qquad 1\le x_1,\ldots, x_s \le H.
\end{split}
\end{equation}

We study $N_{\vec{A}}(H)$ under the following additional assumption on the roots of the characteristic polynomial.

As usual, we say that $\alpha_{j_1},\alpha_{j_2}\in K^{\ast}$ are multiplicatively independent, if
\[
\alpha_{j_1}^{e_1} \alpha_{j_2}^{e_2}=1,\qquad e_1,e_2\in\mathbb Z,
\]
implies $e_1=e_2=0$.

We recall that $\alpha_1,\ldots,\alpha_d$ are the characteristic roots of $\vu$.

\begin{assumption}\label{asmp:onalphas}
We assume that any two distinct elements 
\[
\alpha_{j_1},\alpha_{j_2},\qquad 1\le j_1 \ne j_2\le d,
\]
are multiplicatively independent, and that none of the ratios
\[
\frac{\alpha_{j_1}\alpha_{j_2}}{\alpha_{j_3}\alpha_{j_4}},
\qquad 1\le j_1,j_2,j_3,j_4\le d,\quad \{j_1,j_2\}\ne \{j_3,j_4\},
\]
is a root of unity.
\end{assumption}

 We use similar ideas as in the proof of Theorem~\ref{thm:SystSubgrGen} to prove the following result.

\begin{thm}\label{thm:powersumwithLRS} Let $d\ge 2$ and assume that $\alpha_1,\ldots,\alpha_d$ satisfy Assumption~\ref{asmp:onalphas}. Then, uniformly over $\vec{a_1}\in K\times (K^{\ast})^{s}$ and $\vec a_2\in K^{(s+1)}$, satisfying
\begin{equation}\label{eqn:a2conds d>2}
a_{1,i_1}+a_{1,i_2}\ne 0\quad \mathrm{or} \quad a_{2,i_1}+a_{2,i_2}\ne 0, \quad  \text{for any } 1\le i_1,i_2\le s,
\end{equation}
we have
\[
N_{\vec A}(H)
\ll
\begin{cases}
H^{\lfloor s/3 \rfloor}, & \text{if }~a_{1,0}= a_{2,0}=0,\\
H^{\lfloor (s-1)/3 \rfloor}, & \text{otherwise}.
\end{cases}
\]
\end{thm}

If $\alpha_1,\ldots,\alpha_d\in \overline{\Q}^{\ast}$ are distinct roots of some polynomial $f(T)\in \Q[T]$ that is irreducible over $\Q$, where  $\overline{\Q}$ is the algebraic closure of $\Q$, then the multiplicative independence condition in Assumption~\ref{asmp:onalphas} may be replaced by the following weaker condition:
\begin{equation}
\label{eq:assump}
\alpha_{j_1}\alpha_{j_2}
\quad\text{is not a root of unity for any }1\le j_1,j_2\le d.
\end{equation}

Indeed, let $h$ denote the absolute logarithmic height, see~\cite[Chapter~3]{Lang}. Since the roots $\alpha_1,\ldots, \alpha_d$ are conjugates, we have $h(\alpha_1)=\cdots=h(\alpha_d)$. Moreover since $\alpha_1$ is not a root of unity, we also have $h(\alpha_1)>0$, see~\cite[Proposition~1.1]{Lang}. Now, suppose, $\alpha_{j_1}^{e_1}=\alpha_{j_2}^{e_2}$, for some integers $1\le j_1\ne j_2\le d$ and some integers $e_1,e_2$ not both zero. Then applying $h$, we get $|e_1|=|e_2|$; see the properties mentioned in~\cite[Chapter~3, Section~1]{Lang}. 
Therefore either $\(\alpha_{j_1}/\alpha_{j_2}\)^{e_1}=1$ or $(\alpha_{j_1}\alpha_{j_2})^{e_1}=1$. The first possibility is ruled out by the fact that $\vu$ is non-degenerate one does not need the second condition in Assumption~\ref{asmp:onalphas}, while the second is ruled out by the assumption~\eqref{eq:assump}.

\begin{rem}
Certainly one can combine the arguments in the proofs of Theorem~\ref{thm:SystSubgrGen} and Theorem~\ref{thm:powersumwithLRS} to prove a similar estimate for diagonal systems of polynomial equations 
\begin{align*}
&a_{j,1}u(x_1)^j +\ldots+a_{j,s}u(x_s)^j = a_{j,0}, \qquad j =1, \ldots, k,\\
& \qquad  \qquad  \qquad 1\le x_1,\ldots, x_s  \le H,
\end{align*}
for any arbitrary $k\ge 2$, at least under the non-degeneracy condition on $\vec{A}$ and assuming that $\alpha_1,\ldots, \alpha_d$ are multiplicatively independent. Moreover under the same conditions, one can also ask about the {\it paucity\/} of solutions to the system of equations
\begin{equation*}
\begin{split}
&u(x_1)^j +\ldots+u(x_s)^j = u(x_{s+1})^j +\ldots+u(x_{2s})^j,  \qquad j =1, \ldots, k,\\
& \qquad  \qquad  1\le x_1,\ldots, x_s, x_{s+1},\ldots, x_{2s}  \le H, 
\end{split}
\end{equation*}
with the obvious restriction 
\[
x_1,\ldots, x_s \quad \text{is not a permutation of} \quad  x_{s+1},\ldots, x_{2s}.
\]
\end{rem}

\subsection{Notation} 
We recall the following convention: the notations  $U \ll V$ and $U = O(V)$, are equivalent to $|U|  \le c V$ for some constant $c>0$,
which throughout the paper may depend on $s$, the rank of $\Gamma$ and $d$.  In particular, all our estimates are uniform with respect to the coefficients of the 
diagonal systems of polynomial equations that we consider.

Also, as usual, $\Q$ and $\Z$ denote the sets of rational numbers and integers, respectively. 

\section{Preliminaries}
\subsection{Linear equations in elements of multiplicative subgroups of finite rank}\label{sec:exponentials}
Let $n \ge 1$ be an integer, and $\Gamma$ a subgroup  of $(K^\ast)^n$ of  finite rank $\varrho$. For any
$\gamma_1,\dots,\gamma_n\in K^\ast$, we say that a solution to the equation
\begin{equation}\label{eq:subspace}
\gamma_1g_1+\ldots+\gamma_ng_n =1,
\qquad (g_1,\dots,g_n)\in \Gamma,
\end{equation}
is non-degenerate, if
\[
\sum_{i\in \cI} \gamma_i g_i \neq 0
\]
for every nonempty subset $\cI\subseteq \{1,2,\dots,n\}$, or in other words, if no subsum in the  left-hand side of~\eqref{eq:subspace} vanishes. Then, we have the following estimate due to Amoroso and Viada~\cite[Theorem~6.2]{AV09}.

\begin{lem}\label{lem:non-deg}
For any $\gamma_1,\dots,\gamma_n\in K^\ast$, the number of non-degenerate solutions to~\eqref{eq:subspace} is at most $(8n)^{4n^4(n+\varrho+1)}$.
\end{lem}

Below, we record the following simple consequence, which is used several times in the sequel.

\begin{lem}\label{lem:non-deg_zero-sum}
For any $\gvec=(\gamma_1,\dots,\gamma_n)\in (K^\ast)^n$, there exists a subset
\[
\mathcal C_{\Gamma,\gvec}\subseteq(K^\ast)^n \quad \mathrm{with}\quad
\sharp\, \mathcal C_{\Gamma,\gvec}\le (8n-8)^{4(n-1)^4(n+\varrho)},
\]
such that every non-degenerate solution, to the equation
\[
\gamma_1g_1+\ldots+\gamma_ng_n =0,
\qquad (g_1,\dots,g_n)\in \Gamma,
\]
can be written as
\[
(g_1,g_2,\ldots,g_n)=(g_1,g_1c_2,\ldots,g_1c_n),
\]
for some $(1,c_2,\ldots,c_n)\in \mathcal C_{\Gamma,\gvec}$.
\end{lem}

\begin{proof}
Let 
\[
\widetilde \Gamma =\left\{\left(\frac{g_2}{g_1},\ldots,\frac{g_n}{g_1}\right):
(g_1,\ldots,g_n)\in\Gamma\right\}.
\] Note that $\widetilde \Gamma$ is a subgroup of $(K^\ast)^{n-1}$ with rank at most $\varrho$. Thus, dividing
\[
\gamma_1g_1+\ldots+\gamma_ng_n=0
\]
by $-\gamma_1g_1$, we see that $(g_1,\ldots,g_n)$ is a non-degenerate solution to the above equation if and only if $\(\frac{g_2}{g_1},\ldots, \frac{g_n}{g_1}\)$ is a non-degenerate solution to the equation
\[
-\frac{\gamma_2}{\gamma_1}\widetilde g_2
-\ldots-
\frac{\gamma_n}{\gamma_1}\widetilde g_n=1,\quad (\widetilde g_2,\ldots, \widetilde g_n)\in \widetilde \Gamma.
\]

By Lemma~\ref{lem:non-deg}, the later equation has at most $(8n-8)^{4(n-1)^4(n+\varrho)}$ solutions $(\widetilde g_2,\ldots, \widetilde g_n)$. The proof concludes by setting $\mathcal C_{\Gamma,\gvec}$ to be the set of such corresponding vectors $(1,\widetilde g_2,\ldots, \widetilde g_n)$.
\end{proof}

We also need the following result counting arbitrary solutions, which is essentially~\cite[Corollary~16]{BGKS} (see also~\cite[Lemma~3.2]{ManOstShp}).

\begin{lem} 
\label{lem:LinEq homog}
Let  $\cW \subseteq \Gamma$ be a finite subset of cardinality $W$.  For any  $\gamma_1, \ldots, \gamma_n \in \K^\ast$, the number of
  solutions to
  \[
    \gamma_1 g_1 + \ldots + \gamma_n g_n = 0,\quad g_1, \ldots, g_n \in \cW,
  \]
  is $O\(W^{\fl{n/2}} \)$. 
\end{lem}

In the following sections, we use the above results with $n$ bounded in terms of $s$ and $d$.

\subsection{Ratios of powers in finite sets} 

\begin{lem}\label{lem:x1vsx2}
Let $d\ge 2$ and $\alpha_1,\ldots,\alpha_d\in K^{\ast}$ be such that any two 
\[\alpha_{j_1},\alpha_{j_2},\qquad 1\le j_1< j_2\le d,\]
are multiplicatively independent. Let $\cC \subseteq K^{\ast}$ be a finite set. Then there exists a set $\cZ_{\cC}$ of integers of cardinality at most $d^2\sharp\, \cC$, such that for any integers $x_1,x_2\ge 1$ satisfying
\[
\alpha_{j_1}^{x_1}/\alpha_{j_2}^{x_2}\in \cC\quad \mathrm{for~some}\quad 1\le j_1,j_2\le d,
\]
we have 
\[
x_1-x_2\in \cZ_{\cC}.
\] 
\end{lem} 

\begin{proof}
Fix some $c\in\cC$ and $1\le j_1,j_2\le d$. We obviously have $d^2\sharp\, \cC$ choices
for $(c,j_1,j_2)$.

Suppose there is a solution to
\[
\alpha_{j_1}^{x_1}/\alpha_{j_2}^{x_2}=c,
\]
for some integers $x_1,x_2\ge 1$. If we also have $\alpha_{j_1}^{y_1}/\alpha_{j_2}^{y_2}=c$ for some integers $y_1,y_2\ge 1$, then we get
\[
\alpha_{j_1}^{x_1-y_1}\alpha_{j_2}^{y_2-x_2}=1.\]

If $j_1\ne j_2$, then we have $x_1=y_1$ and $x_2=y_2$. If instead $j_1=j_2$, then $x_1-x_2=y_1-y_2$. This gives the desired conclusion.
\end{proof}

\begin{rem}\label{rem:weakermultindepcond}
In fact, if we have $\alpha_{j_1}^{e_1}=\alpha_{j_2}^{e_2}$ for some positive integers $e_1 \ne e_2$ and some $1\le j_1,j_2\le d$, then the conclusion of Lemma~\ref{lem:x1vsx2} does not hold. Indeed, $\alpha_{j_1}^{x_1}/\alpha_{j_2}^{x_2}=c$ implies $\alpha_{j_1}^{x_1+xe_1}/\alpha_{j_2}^{x_2+xe_2}=c$ for any integer $x\ge 1$.
\end{rem}

\subsection{System of equations with power sums}
We record the following immediate consequence of Newton's identities (connecting elementary symmetric polynomials to power sums), which also underlies the argument of Mordell~\cite[Section~2]{Mord}. 

\begin{lem}\label{lem:VMVTsmalls} Let $s\ge 1$ be an integer. Then for any solution 
\[(c_1,\ldots, c_{s}, c_{s+1},\ldots, c_{2s})\in K^{2s}\] to the system of equations
\[c_1^j +\ldots+c_s^j = c_{s+1}^j +\ldots+c_{2s}^j,  \quad j =1, \ldots s,\]
we have
\[
c_1,\ldots, c_s \quad \text{is a permutation of} \quad c_{s+1},\ldots, c_{2s}.\]
\end{lem}

\section{Proof of Theorem~\ref{thm:SystSubgrGen}}

\subsection{The homogeneous case: $m =0$} 
\label{sec:hom case}
The case $1\le s\le k$ is immediate, since the imposed condition on $\vec{A}$ implies that $M_{\vec{A}}(\cW)=0$. Thus, it remains to consider the case $1\le k\leq s-1$.

We now prove by induction on $k$ that, there exists a finite subset $\cC_{\Gamma,\vec{a}_1,\ldots,\vec{a}_{k}}\subseteq K^{\ast}$, which is closed under taking inverse and containing $1$, 
 of cardinality $\sharp\, \cC_{\Gamma,\vec{a}_1,\ldots,\vec{a}_{k}}\ll 1$ (that is, bounded only in terms of $s$), 
satisfying the following: for any solution $(w_1,\ldots, w_s)$ to~\eqref{eq:GenEqnsSubgr}, there exists a partition
\[
\{1,\ldots,s\}=\bigsqcup_{\newr=1}^{h_k} \cI^{(k)}_{\newr},
\]
such that $\sharp\, \cI^{(k)}_{\newr}\ge k+1$ for each $r=1,\ldots, h_k$ and 
\[w_{i_1}/w_{i_2}\in \cC_{\Gamma,\vec{a}_1,\ldots,\vec{a}_k},\] for any $i_1,i_2\in \cI^{(k)}_{\newr}$ and any $r=1,\ldots,h_k$. The proof then follows immediately, as we have $h_k\le \lfloor s/(k+1) \rfloor$.

\subsubsection*{Basis of induction: $k=1$}\label{sec:k=1powsum}
We argue exactly as in the starting point in the proof of~\cite[Corollary~1.6]{BGKS}. In other words, taking $I=1$ we see that the imposed condition on $\vec{A}$ implies $a_{1,1},\ldots, a_{1,s}\in K^{\ast}$. Then we divide the equation in~\eqref{eq:GenEqnsSubgr} into non-degenerate vanishing sub-sums. Since we have
\[a_{1,1}w_1,\ldots, a_{1,s}w_s\in K^{\ast},\]
clearly each such non-degenerate vanishing sub-sum has at least two terms. Then the proof for the case $k=1$ follows from Lemma~\ref{lem:LinEq homog}. 

\subsubsection*{Inductive step: $k\ge 2$}\label{sec:k ge 2 fgm} 
Let $(w_1,\ldots, w_s)$ be a solution to~\eqref{eq:GenEqnsSubgr}.

Assume that the inductive statement holds for $k-1$, that is,  there exists a partition 
\[
\{1,\ldots,s\}=\bigsqcup_{\newr=1}^{h_{k-1}} \cI^{(k-1)}_{\newr},
\]
such that $\sharp\, \cI^{(k-1)}_{\newr}\ge k$ for each $r=1,\ldots, h_{k-1}$ and 
\begin{equation}
\label{eq:w12}
w_{i_1}/w_{i_2}\in \cC_{\Gamma,\vec{a}_1,\ldots,\vec{a}_{k-1}},
\end{equation} 
for any $i_1,i_2\in \cI^{(k-1)}_{\newr}$ and any $r=1,\ldots,h_{k-1}$.
Fix some index $i_r\in \cI^{(k-1)}_\newr$ for each $r=1,\ldots,h_{k-1}$. Then for any $i\in \cI^{(k-1)}_\newr$ and any $1\le r\le h_{k-1}$, one can write $w_i=c_{i,r} w_{i_\newr}$ for some $c_{i,r}\in \cC_{\Gamma,\vec{a}_1,\ldots,\vec{a}_{k-1}}$. In particular, we have
\begin{equation}
\label{eq:jthsum}
\sum_{i=1}^s a_{j,i}w_i^j=\sum_{r=1}^{h_{k-1}}w_{i_r}^j \sum_{i\in \cI^{(k-1)}_\newr} a_{j,i} c_{i,r}^j =0,\qquad j=1,\ldots, k. 
\end{equation}

Now note that, if the partitions satisfy $\sharp\, \cI^{(k-1)}_{\newr}\ge k+1$ for each $r=1,\ldots, h_{k-1}$, then the proof concludes. Therefore,  since $\sharp\,\cI^{(k-1)}_{\newr}\ge k$, $r=1,\ldots,h_{k-1}$, we may assume that $\sharp\, \cI^{(k-1)}_{\newr}=k$ for some $1\le r\le h_{k-1}$. Let us denote 
\[\cK=\{1\le r \le h_{k-1}:~\sharp\, \cI^{(k-1)}_{\newr}=k\}.\]

Now we take some $\newr_0 \in \cK$. Due to the assumption~\eqref{eq:NonDegen A} on $\vec{A}$, there exists some $1\le j\le k$ (depending on $\vec{A}$, $\cI^{(k-1)}_{r_0}$ and $\{c_{i,r_0}:~i\in \cI^{(k-1)}_{r_0}\}$),  such that
\begin{equation}\label{eq:non-vanishingr0}
\sum_{i\in \cI^{(k-1)}_{\newr_0}} a_{j,i}c_{i,r_0}^{j} \ne 0.
\end{equation} 
Then we decompose the sum in~\eqref{eq:jthsum}, corresponding to $j$ as in~\eqref{eq:non-vanishingr0}, 
into non-degenerate vanishing sub-sums.  Therefore, there is a set $\cK_0 \subseteq \{1,\ldots, h_{k-1}\}$ with 
$r_0 \in \cK_0\cap \cK$  and such that 
\begin{equation}
\label{eq:eqs K0}
\sum_{r \in \cK_0} w_{i_r}^j \sum_{i\in \cI^{(k-1)}_\newr} a_{j,i} c_{i,r}^j =0
\mand \sum_{r \in \cL} w_{i_r}^j \sum_{i\in \cI^{(k-1)}_\newr} a_{j,i} c_{i,r}^j \ne 0
\end{equation}
for any nonempty proper subset $\cL \subseteq \cK_0$.

In particular, by~\eqref{eq:non-vanishingr0} and~\eqref{eq:eqs K0},  there exists some integer 
$\rho(\newr_0)\in \cK_0 \setminus \{\newr_0\}$ such that 
\[\sum_{i\in \cI^{(k-1)}_{\rho(\newr_0)}} a_{j,i}c_{i,\rho(\newr_0)}^{j} \ne 0.\]

We choose one such $\rho(\newr_0)$ for each $r_0\in \cK$. We note that $\rho(\newr_0)$
is not uniquely defined, and the other way around, the same value of $\rho(\newr_0)$ may correspond to 
different values of $r_0$.

Lemma~\ref{lem:non-deg_zero-sum} shows that for some positive integer  $j \le k \le s-1$ the powers $(w_{i_{\rho(r_0)}}/w_{i_{r_0}})^j$ take finitely many values, and, hence so do the ratios $w_{i_{\rho(r_0)}}/w_{i_{r_0}}$ themselves. Therefore, there exists a finite subset $\cC_{\Gamma,\vec{a}_1,\ldots,\vec{a}_{k}}\subseteq K^{\ast}$, which is also closed under taking inverse and  of cardinality $\sharp\, \cC_{\Gamma,\vec{a}_1,\ldots,\vec{a}_{k}} \ll 1$, containing $\cC_{\Gamma,\vec{a}_1,\ldots,\vec{a}_{k-1}}$, such that  
\begin{equation}
\label{eq:rho r0}
w_{i_{\rho(\newr_0)}}/w_{i_{\newr_0}}\in \cC_{\Gamma,\vec{a}_1,\ldots,\vec{a}_{k}}
\end{equation}
 for any $\newr_0\in \cK$. 

Then we denote
\[
\cJ^{(k)}_{\newr} =
\begin{cases}
\cI^{(k-1)}_{{\newr}}\cup \cI^{(k-1)}_{\rho(\newr)}, & \newr\in \cK,\\
\cI^{(k-1)}_{\newr}, & \newr\notin \cK.
\end{cases}\]

In particular, since for $ \newr\in \cK$ we have chosen $\rho(\newr) \ne \newr$  we now conclude that 
$\cI^{(k-1)}_{{\newr}}\cap \cI^{(k-1)}_{\rho(\newr)} = \varnothing$. 
Hence, 
 $\sharp\, \cJ^{(k)}_{\newr}\ge 2k\ge k+1$ for any $\newr\in \cK$ and also, by the definition of $\cK$,
 we have $\sharp\, \cJ^{(k)}_{\newr}\ge k+1$ for any $r\notin \cK$. Moreover, we can write
\begin{equation}\label{eq:union}
\{1,\ldots, s\}=\bigcup_{1\le r\le h_{k-1}} \cJ^{(k)}_{\newr},
\end{equation} 

By~\eqref{eq:w12} and~\eqref{eq:rho r0},  if $(w_1,\ldots, w_s)$ is a solution to~\eqref{eq:GenEqnsSubgr}, then we have 
\begin{equation}
\label{eq:w12 Jr}
w_{i_1}/w_{i_2}\in \cC_{\Gamma,\vec{a}_1,\ldots,\vec{a}_k},
\end{equation}
for any $i_1, i_2 \in \cJ^{(k)}_{\newr}$ and any $1\le r\le h_{k-1}$.

Next, we show how to modify the union~\eqref{eq:union} into the desired partition.

First of all, for any integer $n\ge 1$, let us denote the product-set 
\[
\cC_{\Gamma,\vec{a}_1,\ldots,\vec{a}_{k}}^n=\{c_1\ldots c_n: c_i\in \cC_{\Gamma,\vec{a}_1,\ldots,\vec{a}_{k}},\quad i=1,\ldots, n\}.
\]
Since $1\in \cC_{\Gamma,\vec{a}_1,\ldots,\vec{a}_{k}}$, note that
\[
\cC_{\Gamma,\vec{a}_1,\ldots,\vec{a}_{k}}^m \subseteq \cC_{\Gamma,\vec{a}_1,\ldots,\vec{a}_{k}}^n,\qquad  1\le m\le n.
\]

Let us now consider the graph $\cG$ with vertices $\{1,....,h_{k-1}\}$, where there is an edge between two vertex $r_1$ and $r_2$, if and only if, 
\[
\cJ^{(k)}_{r_1}\cap \cJ^{(k)}_{r_2} \neq \varnothing.
\]

Now, let $\cG_1,\ldots, \cG_{h_k}$ be the connected components of $\cG$, for some integer $1\le h_k\le h_{k-1}$. Then we denote
\[
\cI^{(k)}_{r^*}=\bigcup_{r\in \cG_{r^*}} \cJ^{(k)}_{r},\quad 1\le r^*\le h_{k}.
\]

It is clear that, 
\[
 \cI^{(k)}_{r_1^*}\cap \cI^{(k)}_{r_2^*}=\varnothing,\quad 
 \mathrm{for~any}\quad 1\le r^*_1\ne r^*_2\le h_k.
 \]
In particular, we get a partition
\[
\{1,\ldots,s\}=\bigsqcup_{r^*=1}^{h_k} \cI^{(k)}_{\newr^*}.
\]
Moreover, we also have
\[\cI^{(k)}_{r^*}\ge k+1,\quad r^*=1,\ldots, h_k.\]

Then (as we trivially have $h_{k-1}\le s$), it is enough to show that 
\[
w_{i_1}/w_{i_2}\in \cC_{\Gamma,\vec{a}_1,\ldots,\vec{a}_k}^{h_{k-1}},
\] 
for any $i_1, i_2 \in \cI^{(k)}_{r^*}$ and any $1\le r^*\le h_{k}$.

To prove the claim above, take some $1\le r^*\le h_{k}$ and some $i_1, i_2 \in \cI^{(k)}_{r^*}$. Let $r_1,r_2 \in \cG_{r^*}$ be the vertices, such that
\[
i_1\in \cJ^{(k)}_{r_1} \mand i_2\in \cJ^{(k)}_{r_2}.
\]
Since $\cG_{r^*}$ is connected and has at most $h_{k-1}$ vertices, there exists a simple (that is, without loops) path in $\cG_{r^*}$ of length at most $h_{k-1}-1$ from $r_1$ to $r_2$, where for each of at most $h_{k-1}-1$ steps, from $i^\circ \in \cJ^{(k)}_{u}$ to $i^\bullet \in \cJ^{(k)}_{v}$, we compute  
\[
\frac{w_{i^\circ}}{w_{^\bullet}}=\frac{w_{i^\circ}}{w_{i}}\cdot \frac{w_{i}}{w_{^\bullet}} 
\]
for some $i \in \cJ^{(k)}_{u}  \cap  \cJ^{(k)}_{v}$.  Using~\eqref{eq:w12 Jr}, we conclude the proof.

\subsection{The inhomogeneous case: $m \ge 1$} 
We now prove the desired estimate above by induction on $s$. The base case $s=1$ is trivial. Therefore, we assume that the inductive statement holds for $s-1$ and now prove the estimate for $s$.

Note that the case $s+1\le m$ is immediate, since the imposed condition on $\vec{A}$ implies that $M_{\vec{A}}(\cW)=0$. Thus, it remains to consider the case $1\le m\leq s$.

As we have shown in Section~\ref{sec:hom case}, from the first $m-1$ equations,  which form a homogeneous system, we obtain a finite subset $\cC_{\Gamma,\vec{a}_1,\ldots,\vec{a}_{m-1}}\subseteq K^{\ast}$ of cardinality $\sharp\, \cC_{\Gamma,\vec{a}_1,\ldots,\vec{a}_{m-1}}\ll 1$, satisfying the following: 
for any solution $(w_1,\ldots, w_s)$ to~\eqref{eq:GenEqnsSubgr}, there exists a partition
\[
\{1,\ldots,s\}=\bigsqcup_{\newr=1}^{h} \cI_{\newr},
\]
such that $\sharp\, \cI_{\newr}\ge m$ for each $r=1,\ldots, h$ and 
\[w_{i_1}/w_{i_2}\in \cC_{\Gamma,\vec{a}_1,\ldots,\vec{a}_{m-1}},\] for any $i_1,i_2\in \cI_{\newr}$ and any $r=1,\ldots,h$. 

Now let $(w_1,\ldots, w_s)$ be a solution to~\eqref{eq:GenEqnsSubgr} and fix some index $i_r\in \cI_\newr$ for each $r=1,\ldots, h$. Then for any $i\in \cI_\newr$ and any $1\le r\le h$, one can write $w_i=c_{i,r} w_{i_\newr}$ for some $c_{i,r}\in \cC_{\Gamma,\vec{a}_1,\ldots,\vec{a}_{m-1}}$. In particular from the $m$th equation, we have
\[
\sum_{i=1}^s a_{m,i}w_i^m=\sum_{r=1}^{h}w_{i_r}^m~\sum_{i\in \cI_\newr} a_{m,i} c_{i,r}^m =a_{m,0}. 
\]

Let us denote 
\[
\cR=\left\{1\le r \le h:~\sum_{i\in \cI_\newr} a_{m,i} c_{i,r}^m \ne 0\right\}.
\]
Since $a_{m,0}\ne 0$, clearly we have $\cR \ne \varnothing$. In particular, we can write  
\begin{equation}
\label{eq:eq R}
\sum_{r\in \cR} w_{i_r}^m~\sum_{i\in \cI_\newr} a_{m,i} c_{i,r}^m =a_{m,0}.
\end{equation}

As in the homogeneous case, we observe that we have only finitely many roots of unity of order at most $m\le s$.
Since again $a_{m,0}\ne 0$, following the proof of~\cite[Lemma~3.3]{ManOstShp}, we split the left-hand side of~\eqref{eq:eq R} into non-degenerate vanishing subsums and prove that there exists a finite subset 
$\cC_{\Gamma,\vec{a}_1,\ldots,\vec{a}_{m}}\subseteq K^{\ast}$ containing $\cC_{\Gamma,\vec{a}_1,\ldots,\vec{a}_{m-1}}$ of cardinality $\sharp\, \cC_{\Gamma,\vec{a}_1,\ldots,\vec{a}_{m}}\ll 1$, such that for some $t \in \cR$ we have 
\[
w_{i_t}\in \cC_{\Gamma,\vec{a}_1,\ldots,\vec{a}_{m}},
\]
and for any $i\in\cI_i$, $i\ne i_t$, $w_i=c_{i,t}w_{i_t}\in \cC^2_{\Gamma,\vec{a}_1,\ldots,\vec{a}_{m}}$.
Moreover, since $\sharp\, \cI_{r}\ge m$ for all $r$, we have 
\begin{equation}\label{eqn:m many good w}
\sharp\, \{1\le i\le s:~w_i\in \cC^2_{\Gamma,\vec{a}_1,\ldots,\vec{a}_{m}}\}\ge 
\sharp\, \cI_t \ge m.
\end{equation}

In particular, the proof concludes immediately if we have $s=m$. Therefore, we may assume that $s-m\ge 1$. Moreover, fixing $w_i$ with $i$ as in~\eqref{eqn:m many good w}, we obtain $O(1)$ systems of equations of the form
\begin{align*}
&a_{j,1}v_1^j +\ldots+a_{j,s-m}v_{s-m}^j =\widetilde a_{j,0}, \qquad j =1, \ldots k,\\
& \qquad  \qquad  \qquad v_1,\ldots, v_{s-m}  \in \cW, 
\end{align*}
for some fixed
\[
\widetilde a_{1,0},\ldots,\widetilde a_{k,0}\in K^*.
\]

Note that the tuple of coefficients appearing in any of the above 
systems of equations is non-degenerate. In particular, from Section~\ref{sec:hom case}, we get the contribution $O\(W^{\lfloor (s-m)/(k+1) \rfloor}\)$
in the case $\widetilde a_{1,0}=\ldots=\widetilde a_{k,0}=0$. Otherwise, since the system of equations has $1\le s-m\le s-1$ variables, by induction, each of the remaining cases of $(\widetilde a_{1,0},\ldots,\widetilde a_{k,0})  \in \(K^*\)^k \setminus \{0\}$ contributes at most $O\(W^{\lfloor (s-m-1)/(k+1) \rfloor}\)$, 
and the proof concludes.

\section{Proof of Theorem~\ref{thm:VMVT}}

\subsection{Reduction to irreducible systems}
For each $t=1,\ldots, s$, we count the number of solutions to~\eqref{eq:VMVT}, such that $s-t$ many values of $w_1,\ldots,w_s$ coincides with $s-t$ many values of $w_{s+1},\ldots, w_{2s}$.  Then, 
relabeling the indices, our  system of equations becomes of the form
\begin{equation}\label{eq:VMVTreduced}
\begin{split}
&w_1^j+\ldots+w_t^j
= w_{t+1}^j+\ldots+w_{2t}^j,
\qquad j=1,\ldots,k,\\
& \qquad  \qquad  w_1,\ldots, w_{t}, w_{t+1},\ldots, w_{2t}  \in \cW,
\end{split}
\end{equation}
with the more stringent restriction 
\begin{equation}\label{eq:nothingcommon}
\{w_1,\ldots,w_t\}\cap\{w_{t+1},\ldots,w_{2t}\}=\varnothing.
\end{equation}

We denote by $J_{t,k}^{\ast}(\cW)$ the number of solutions to the system of such equations, which we call {\it irreducible\/}. 

By~\eqref{eq:Jss}, we have 
\[J_{t,k}^{\ast}(\cW)=J_{t,k}(\cW)=0,\qquad t=1,\ldots, k.\]
In particular, we can write
\begin{equation}\label{eq:SumOfJ'}
    J_{s,k}(\cW)= \sum_{t=1}^{s} W^{s-t} J_{t,k}^{\ast}(\cW)=\sum_{t=k+1}^{s} W^{s-t} J_{t,k}^{\ast}(\cW).
\end{equation}

Hence to prove Theorem~\ref{thm:VMVT} it is enough to 
show that 
\begin{equation}\label{eq:J'}
J_{t,k}^{\ast}(\cW)\ll W^{\lfloor 2t/(k+1)\rfloor},
\end{equation}
for any $1\le k< t\le s$. To prove the claim we argue by induction on $k$,  exactly as in the proof of Theorem~\ref{thm:SystSubgrGen}. Let us provide a brief sketch of the argument.

By induction on $k$, we show that there exists a finite subset $\cC_{\Gamma}\subseteq K^{\ast}$ of cardinality $\sharp\, \cC_{\Gamma}\ll 1$, satisfying the following: for any solution $(w_1,\ldots,w_t,w_{t+1},\ldots, w_{2t})$ to~\eqref{eq:VMVTreduced}, there exists a partition
\[
\{1,\ldots,2t\}=\bigsqcup_{\newr=1}^{h_k} \cI^{(k)}_{\newr},
\]
such that $\sharp\, \cI^{(k)}_{\newr}\ge k+1$ and $w_{i_1}/w_{i_2}\in \cC_{\Gamma}$, for any $i_1,i_2 \in \cI^{(k)}_{\newr}$ and any $r=1,\ldots,h_k$. If we show this, then~\eqref{eq:J'} follows.

\subsection{Basis of induction: $k=1$}
The proof follows by the same argument as in  Section~\ref{sec:k=1powsum}.

\subsection{Inductive step: $k\ge 2$} For each $j=1,\ldots, k$, let us denote
\[
a_{j,i}=1,\quad i\in \{1,\ldots, t\} \mand a_{j,i}=-1, \quad i\in \{t+1,\ldots, 2t\}.
\]

Let  $(w_1,\ldots,w_t,w_{t+1},\ldots,w_{2t})$ be a solution to~\eqref{eq:VMVTreduced}. Assume that the inductive statement holds for $k-1$, that is,  there exists a partition  
\[
\{1,\ldots,2t\}=\bigsqcup_{\newr=1}^{h_{k-1}} \cI^{(k-1)}_{\newr},
\]
such that $\sharp\, \cI^{(k-1)}_{\newr}\ge k$ for each $r=1,\ldots, h_{k-1}$. 
Fix some index $i_r\in \cI^{(k-1)}_\newr$ for each $r=1,\ldots,h_{k-1}$.
Then for any $i\in \cI^{(k-1)}_\newr$, one can write $w_i=c_{i,r} w_{i_\newr}$ for some $c_{i,r}\in \cC_{\Gamma}$, with $\cC_{\Gamma}$ as above. In particular, we have
\[\sum_{r=1}^{h_{k-1}}\(\sum_{i\in \cI^{(k-1)}_\newr} a_{j,i} c_{i,r}^j\)w_{i_r}^j=0,\qquad j=1,\ldots, k.\]

If $\sharp\, \cI^{(k-1)}_{\newr}\ge k+1$ for all $r=1,\ldots,h_{k-1}$, then we are done. Otherwise, let 
$1\le r_0\le h_{k-1}$ be an integer such that $\sharp\, \cI^{(k-1)}_{\newr_0}=k$. Following the same argument as in the proof of Theorem~\ref{thm:SystSubgrGen}, it is enough to show that there exists some $1\le j\le k$ (depending on $\cI_{r_0}$ and $\{c_{i,r_0}:~i\in \cI_{r_0}\}$), such that
\begin{equation}\label{eq:jthsumnonzero}
\(\sum_{i\in \cI^{(k-1)}_{\newr_0}} a_{j,i} c_{i,r}^j\)w_{i_{r}}^j=\sum_{i\in \cI^{(k-1)}_{\newr_0}} a_{j,i} w_i^j\ne 0.
\end{equation}

Let us now set 
\[\cI_1=\cI^{(k-1)}_{\newr_0} \cap \{1,\ldots, t\} \mand \cI_2= \cI^{(k-1)}_{\newr_0} \cap \{t+1,\ldots, 2t\}.\] Suppose~\eqref{eq:jthsumnonzero} does not hold, then we have
\begin{equation}\label{eq:sum wj two sides}
\sum_{i\in \cI_1}  w_i^j=\sum_{i\in \cI_2}  w_i^j,\qquad j=1,\ldots,k.
\end{equation}

Without loss of generality, we may assume that $\sharp\, \cI_1\ge \sharp\, \cI_2$. If $\sharp\, \cI_1> \sharp\, \cI_2$, then adding sum of powers of zeroes to the right side above, Lemma~\ref{lem:VMVTsmalls} shows that $0\in \{w_i:i\in \cI_1\}$, which is not possible. Moreover if we have $\sharp\, \cI_1= \sharp\, \cI_2$, then Lemma~\ref{lem:VMVTsmalls} implies 
\[\{w_i:~i\in \cI_1\}=\{w_i:~i\in \cI_2\},\]
which contradicts~\eqref{eq:nothingcommon}.

Now following the same argument as in the proof of Theorem~\ref{thm:SystSubgrGen}, we can construct the desired partition of $\{1,\ldots,2t\}$. The proof of~\eqref{eq:J'} concludes.

\subsection{Concluding the proof}

Combining~\eqref{eq:SumOfJ'} and~\eqref{eq:J'}, we have
\begin{align*}
J_{s,k}(\cW)&= \sum_{t=k+1}^{s} W^{s-t} J_{t,k}^{\ast}(\cW)\le \sum_{t=k+1}^{s} W^{s-t} \cdot W^{\lfloor 2t/(k+1)\rfloor}\\
&\le \sum_{t=k+1}^{s} W^{s-\frac{t(k-1)}{(k+1)}}\ll W^{s-k+1},
\end{align*}
and the proof concludes.

\section{Proof of Theorem~\ref{thm:lrslineqn}} 
\label{sec:GenEq} 

\subsection{Generalisation}
For any 
\[
\gvec=\(\gamma_{i,j}\)_{\substack{1\le i\le s\\ 1\le j\le d}}\in K^{sd}
\quad \mathrm{and}\quad a_0\in K,
\]
let $S_{\gvec,a_0}(H)$ denote the number of solutions to the equation
\begin{equation}\label{eq:sa}
\sum_{i=1}^s\sum_{j=1}^d \gamma_{i,j}\alpha_j^{x_i}=a_0,
\quad 1\le x_1,\ldots,x_s\le H.
\end{equation}

To prove Theorem~\ref{thm:lrslineqn} we prove the following stronger result, which is also needed in its full strength for the proof of Theorem~\ref{thm:powersumwithLRS} in Section~\ref{sec:powersumsLRS}. 

\begin{thm}\label{thm:anycombination}
Assume that neither any of $\alpha_1,\ldots,\alpha_d$ nor any of the ratios $\alpha_{j_1}/\alpha_{j_2}$ with $1\leq j_1\neq j_2\leq d$, is a root of unity. Then uniformly over $\gvec=(\gamma_{i,j})\in K^{sd}$ such that, for every $i=1,\ldots,s$, at least one of the $\gamma_{i,1},\ldots,\gamma_{i,d}$ is nonzero, we have
    \begin{equation}\label{eq:s0}
    S_{\gvec,0}(H)\ll H^{\lfloor s/2 \rfloor}, 
    \end{equation}
   and
    \begin{equation}\label{eq:sa0}
    S_{\gvec,1}(H)\ll H^{\lfloor (s-1)/2 \rfloor}.
    \end{equation}
    \end{thm}

As in~\cite[Section~3.2]{ManOstShp}, we prove Theorem~\ref{thm:anycombination}
by induction on $s$.  

\subsection{Basis of induction: $s=1$}\label{sec:s0inductionbase}
The proof for the case $a_0=0$ follows from the Skolem-Mahler-Lech theorem. Indeed, from the conditions on $\alpha_j$ and $\gamma_{1,j}$, $j=1,\ldots,d$,  $u(x_1)=\sum_{j=1}^d\gamma_{1,j} \alpha_j^{x_1}$ is a non-zero non-degenerate linear recurrence, and hence we have $S_{\gvec,0}(H)=O(1)$.

For $a_0=1$, we have the equation 
\[
\gamma_{1,1} \alpha_1^{x_1}+\ldots+\gamma_{1,d} \alpha_d^{x_1}=1.
\] 
Then, if no proper subsum of the left-hand side vanishes, Lemma~\ref{lem:non-deg} applies, since none of $\alpha_1,\ldots,\alpha_d$ is a root of unity. Otherwise, we can split the sum above into two subsums, where one of the sum is $0$ and the other sum (with at most $d-1$ terms) is $1$. 
The proof then follows by induction on $d$ (the case $d=1$ being immediate), and applying the Skolem-Mahler-Lech theorem again. We conclude in this case as well that $S_{\gvec,1}(H)=O(1)$.

\subsection{Inductive step: $s\ge 2$} 
Now we prove~\eqref{eq:s0} for any $s\ge 2$. Again by  the Skolem--Mahler--Lech theorem,  we have
\[
\gamma_{s,1} \alpha_1^{x_s}+\ldots+\gamma_{s,d}\alpha_d^{x_s}=0,
\]
for finitely many $x_s\ge 1$, because at least  one of the $\gamma_{s,1},\ldots, \gamma_{s,d}$ is non-zero. By induction, each such finitely many $x_s$ contributes only $\ll H^{\lfloor (s-1)/2 \rfloor}$ solutions $(x_1,\ldots, x_s)$ to~\eqref{eq:sa} for $a_0=0$. Then for any remaining fixed $1\le x_s\le H$, we have the equation 
   \begin{align*}
\sum_{i=1}^{s-1}\sum_{j=1}^d \gamma_{i,j}\alpha_j^{x_i}&=-\(\gamma_{s,1} \alpha_1^{x_s}+\ldots+\gamma_{s,d}\alpha_d^{x_s}\)\ne 0,\\
&  1\le x_1,\ldots, x_{s-1} \le H.
\end{align*}

By induction, we have $O\(H^{\lfloor (s-2)/2\rfloor}\)$ solutions $(x_1,\ldots, x_{s-1})$ to the equation above. Since there are at most $H$ possibilities of $x_s$, the proof of~\eqref{eq:s0} concludes.

Now we prove~\eqref{eq:sa0} for any $s\ge 2$. Consider the set of indices 
\[\cI(\gvec)=\left\{(i,j):~\gamma_{i,j}\ne 0,~1\le i\le s,~1\le j\le d\right\}.\]

Let $T$ denote the number of solutions to~\eqref{eq:sa} for $a_0=1$. For each such solution $\vx=(x_1,\ldots, x_{s})$, we associate a subset $\cJ(\vx) \subseteq \cI(\gvec)$ 
with the largest cardinality, such that
\[
\sum_{(i,j)\in \cJ(\vx)} \gamma_{i,j}\alpha_j^{x_i}=0.
\]

For each subset $\cI \subseteq I(\gvec)$, let $T_\cI$ be the number of solutions $\vx$ to~\eqref{eq:sa} with $a_0=1$ and $\cJ(\mathbf{x})=\cI$. Then we can find a subset $\cJ\subseteq  \cI(\gvec)$,
such that
\[
T_\cJ=\max_{\cI\subseteq \cI(\gvec)} T_\cI.\]
In particular, it is now enough to estimate $T_\cJ$, as we can write
\[
T= \sum_{\cI \subseteq \cI(\gvec)}   T_\cI
 \ll  T_\cJ.
\]

For any solution $\vx$ to~\eqref{eq:sa} with $a_0=1$ and $\cJ(\vx)=\cJ$, we can split~\eqref{eq:sa} into two parts: 
\begin{equation}\label{eq:deg}
\sum_{(i,j)\in \cJ} \gamma_{i,j}\alpha_j^{x_i} =0,
\end{equation}
and
\begin{equation}\label{eq:nondeg}
\sum_{(i,j)\in \cI(\gvec)\setminus \cJ} \gamma_{i,j}\alpha_j^{x_i} =1.
\end{equation}

Note that no subsum on the left-hand side of~\eqref{eq:nondeg} vanishes, by the maximality property 
of  $\cJ(\vx) =\cJ$. Thus, Lemma~\ref{lem:non-deg} shows that the number of solutions to~\eqref{eq:nondeg} is $O(1)$.

Note that for $a_0 = 1$ we have $\cJ\ne \cI(\gvec)$, and hence~\eqref{eq:nondeg} involves at least one of the variables $x_i$ for some $1\le i\le s$. Without loss of generality, assume that~\eqref{eq:nondeg} contains the variable $x_{s}$. Then Lemma~\ref{lem:non-deg} implies that the number of possible $x_{s}$ is $O(1)$, again because none of the $\alpha_1,\ldots, \alpha_d$ is a root of unity. For each such fixed $x_{s}$, the equation~\eqref{eq:deg} becomes an equation of the form
\begin{equation}\label{eq:s-2variables}
\sum_{\substack{1\le i\le s-1\\ 1\le j\le d}} \gamma_{i,j}\alpha_j^{x_i}=\widetilde a_0,\quad 1\le x_1,\ldots, x_{s-1} \le H,
\end{equation}
for some $\widetilde a_0 \in K$ depending on $x_{s}$ and $\cJ$.   By induction, the number of solutions $(x_1,\ldots,x_{s-1})$ to~\eqref{eq:s-2variables} is $O\(H^{\lfloor (s-1)/2\rfloor}\)$, and the proof concludes.

\section{Proof of  Theorem~\ref{thm:powersumwithLRS}} \label{sec:powersumsLRS}

\subsection{Preliminaries} We prove the result by induction on $s$. The case $s=1$ follows as in Section~\ref{sec:s0inductionbase} using Lemma~\ref{lem:non-deg}.  We now assume that it 
holds for systems of equations with at most $s-1$ variables.

\subsection{The homogeneous case:  $\(a_{1,0}, a_{2,0}\)=(0,0)$} \label{sec:partition LRS}
Let $(x_1,\ldots,x_s)$ be a solution to the first equation in~\eqref{eq:geneqnsLRS}, then we can write
\begin{equation}\label{eqn:EXPANDED}
\sum_{\substack{1\le i\le s\\ 1\le j\le d}} \gamma_{i,j}\alpha_j^{x_i}=0,\qquad 1\le x_1,\ldots, x_{s} \le H,\ \gamma_{i,j}=a_{1,i}\beta_j,
\end{equation}
where $\beta_1,\ldots,\beta_d\in K^\ast$ are defined as in~\eqref{eq:LRS char roots} and~\eqref{eq:char roots}. We divide this equation into non-degenerate vanishing sub-sums, as in Section~\ref{sec:k=1powsum}.

Now we associate a graph $\cG(\vx)$ with $s$ vertices $1,\ldots,s$, such that there is an edge between $i_1$ and $i_2$, if and only if, $x_{i_1}$ and $x_{i_2}$ appear in one of the same non-degenerate vanishing sub-sum. Note that $\cG(\vx)$ having an isolated vertex, say some $i$, corresponds to a non-degenerate vanishing sub-sum of the form
\[
\sum_{j\in \cJ} \gamma_{i,j}\alpha_j^{x_i}=0
\]
for some $\cJ\subseteq \{1,\ldots,d\}$. Then by the Skolem--Mahler--Lech theorem, there exists a  finite set $\cZ_{\vu,\vec{a}_1}$ of integers of cardinality $\sharp\, \cZ_{\vu,\vec{a}_1}\ll 1$ (that is, bounded only in terms of $s$ and $d$), such that $x_i\in \cZ_{\vu,\vec{a}_1}$.

Now let $i_1$ and $i_2$ be two vertices belong to some same connected component of $\cG(\vx)$. As per the assumption, we have $a_{1,i}\in K^{\ast}$ for each $i=1,\ldots, s$. In particular by Lemma~\ref{lem:non-deg_zero-sum} applied to each of the non-degenerate vanishing sub-sums of~\eqref{eqn:EXPANDED}, there exists a finite subset $\cC_{\vu, \vec{a}_1}\subseteq K^{\ast}$ of cardinality $\sharp\, \cC_{\vu, \vec{a}_1} \ll 1$, such that 
\[\alpha_{j_1}^{x_{i_1}}/\alpha_{j_2}^{x_{i_2}}\in \cC_{\vu,\vec{a}_1},\quad \mathrm{for~some}\quad 1\le j_1, j_2\le d.\]
Then, applying Lemma~\ref{lem:x1vsx2}, there exists a finite subset $\cY_{\vu,\vec{a}_1}$ containing $\cZ_{\vu,\vec{a}_1}$, of cardinality $\sharp\, \cY_{\vu,\vec{a}_1}\ll 1$, such that
\[x_{i_1}-x_{i_2}\in \cY_{\vu,\vec{a}_1}.\]

Simply considering the partition determined by connected components of a graph, we deduce that for any solution $\vx=(x_1,\ldots,x_s)$ to~\eqref{eqn:EXPANDED}, there exists a partition
\[
\{1,\ldots,s\}=\bigsqcup_{\newr=1}^{h} \cI_{\newr},
\]
such that for any $r$ with $\sharp\, \cI_{\newr}\ge 2$, we have
\begin{equation}\label{eq:xi1 - xi2}
x_{i_1}-x_{i_2}\in \cY_{\vu,\vec{a}_1},\quad \mathrm{for~any}\quad i_1,i_2\in \cI_{\newr}.
\end{equation}
In addition, for any $r$ with $\sharp\, \cI_{\newr}= 1$, we have 
\[x_i\in \cY_{\vu,\vec{a}_1},\quad \mathrm{for~any}\quad i\in \cI_r.\]

Now, let $T$ denote the number of solutions to the system of equations
\begin{align*}
&a_{1,1}u(x_1)+\ldots+a_{1,s}u(x_s) = a_{1,0}=0, \\
&a_{2,1}u(x_1)^2+\ldots+a_{2,s}u(x_s)^2 =  a_{2,0}=0, 
\end{align*}
where $1\le x_1,\ldots, x_s \le H$. If $x_i\in \cY_{\vu,\vec{a}_1}$ for some $1\le i\le s$,  then, by induction, we have
\[
T\ll  H^{\lfloor (s-1)/3 \rfloor}.
\]

Thus, it is enough for us to fix a partition  
\begin{equation}\label{eq: part 1..s}
\{1,\ldots,s\}=\bigsqcup_{\newr=1}^{h} \cI_{\newr},
\end{equation}
such that \begin{equation}\label{eq: IR ge 2}
\sharp\, \cI_{\newr} \ge 2,\qquad r=1,\ldots, h,
\end{equation}
and estimate the number of solutions $\vx=(x_1,\ldots,x_s)$ to~\eqref{eq:geneqnsLRS}, corresponding to this fixed partition.
Observe that,  by~\eqref{eq:xi1 - xi2}, one has
\[
T \ll H^h.
\]
 Certainly, we need to consider the case that there exists some $1\le r_0\le h$ such $\sharp\, \cI_{r_0}=2$, as otherwise we already have $h\le \lfloor s/3 \rfloor$  and the result is immediate.

Without loss of generality,  we can assume that $\cI_{r_0}=\{1,2\}$. By the definition of the associated graph, we know that each non-degenerate vanishing subsum in~\eqref{eqn:EXPANDED} containing either $x_1$ or $x_2$, or both of them, is precisely of the form
\[
a_{1,1}\(\sum_{j_1\in \cJ_1} \beta_{j_1}\alpha_{j_1}^{x_1}\)+a_{1,2}\(\sum_{j_2\in \cJ_2}\beta_{j_2}\alpha_{j_2}^{x_2}\)=0,
\]
for some (may be empty) subsets $\cJ_1$ and $\cJ_2$ of $\{1,\ldots,d\}$. In fact, any two such sets $\cJ_1$ (and any two such
sets $\cJ_2$, as well) are pairwise disjoint. Thus, summing over all such $\cJ_1$ and $\cJ_2$, we obtain
\[
a_{1,1}\(\sum_{1\le j \le d} \beta_{j}\alpha_{j}^{x_1}\)+a_{1,2}\(\sum_{1\le j\le d}\beta_{j}\alpha_{j}^{x_2}\)=0.
\]

Now by~\eqref{eq:xi1 - xi2}, we have $x_2-x_1=z$ for some $z\in \cY_{\vu,\vec{a}_1}$. In particular, we can write
\begin{equation}\label{eq:eqn with x1 and z}
\sum_{1\le j\le d}\beta_j(a_{1,1}+ a_{1,2}\alpha_j^{z})\, \alpha_{j}^{x_1}=0.
\end{equation}

Let us first assume that $a_{1,1}+a_{1,2}\ne 0$. Then we have
\begin{equation}\label{eq:one nonzero coeff}
\(a_{1,1}+a_{1,2}\alpha_{j}^z\)_{1\le j\le d}\ne 0.
\end{equation}  
Indeed, as per the assumption, the above is true for $z=0$. If $z\ne 0$ and the above does not hold, then since $d\ge 2$, we get $\alpha_{1}^{z}=\alpha_{2}^{z}$, which contradicts  the fact that the sequence $\vec{u}$ is non-degenerate.

Since we have~\eqref{eq:one nonzero coeff}, applying Theorem~\ref{thm:anycombination} for $s=1$ to \eqref{eq:eqn with x1 and z}, there are at most finitely many solutions $x_1$ to \eqref{eq:eqn with x1 and z}, and thus finitely many solutions $x_2$ as well.

Therefore, it is now enough to assume that the partition satisfies the following: for any $1\le r\le h$ with $\sharp\, \cI_r=2$, we have 
\begin{equation}\label{eq:sum 0}
a_{1,i}+a_{1,i^{\star}}=0,\quad \mathrm{where}\quad \cI_r=\{i,i^{\star}\}.
\end{equation}

Now for each $1\le r\le h$ we fix some $i_r\in \cI_r$ and for each $i\in \cI_{r}$, we write $x_i=x_{i_r}+z_{i,r}$, for some $z_{i,r}\in \cY_{\vu,\vec{a}_1}$.

Using the representation~\eqref{eq:LRS char roots}, we have
\begin{equation}\label{eq:2nd eq beta_i LRS}
a_{2,1}u(x_1)^2+\ldots+a_{2,s}u(x_s)^2 = \sum_{1\le i\le s}a_{2,i} \sum_{1\le j_1,j_2\le d} \beta_{j_1} \beta_{j_2} \(\alpha_{j_1} \alpha_{j_2}\)^{x_i}.
\end{equation}
Therefore, we obtain at most $(\sharp\,  \cY_{\vu,\vec{a}_1})^{s}\ll 1$ different equations of the form
\begin{equation}\label{eq:sum with bj1 bj2}
     \sum_{1\le \newr\le h}~\sum_{1\le j_1,j_2\le d} b_{\newr,j_1,j_2} \(\alpha_{j_1} \alpha_{j_2}\)^{x_{i_\newr}}=0,\quad 1\le x_{i_1},\ldots,x_{i_h}\le H,
\end{equation}    
for some $b_{\newr,j_1,j_2} \in K$, $\newr=1,\ldots, h$ and $j_1,j_2=1,\ldots, d$. 

Let us denote
\[\ell=\sharp\,\left\{1\le \newr\le h:~\(b_{\newr,j_1,j_2}\)_{1\le j_1,j_2\le d}\ne 0 \right\}.
\]
Since $\alpha_1,\ldots,\alpha_d$ satisfy the conditions of Assumption~\ref{asmp:onalphas}, by Theorem~\ref{thm:anycombination}, the number of solutions to~\eqref{eq:sum with bj1 bj2} is at most
\begin{equation}\label{eq:bounds for second eqn}
O\left(H^{h-\ell+\left\lfloor \ell/2 \right\rfloor}\right)
=
O\left(H^{h-\left\lceil \ell/2 \right\rceil}\right).
\end{equation}

Let us now consider 
$\cK=\{1\le r\le h:~\sharp\, \cI_r=2\}$ and denote $k=\sharp\, \cK$. We claim that
\begin{equation}\label{eq:ell vs k}
h\le  \lfloor (s+k)/3\rfloor \mand \ell \ge k.
\end{equation}

Note that the first inequality in~\eqref{eq:ell vs k}  follows trivially from~\eqref{eq: part 1..s} and~\eqref{eq: IR ge 2}. 

To prove the second inequality in~\eqref{eq:ell vs k}, for simplicity, let us write  $\cI_r$, $r\in \cK$, as
\[\{i_{1},i_{1}^\star\},\ldots, \{i_{k},i_{k}^\star\}.
\] 

Recall that for each $\newr=1,\ldots,k$, we have 
$x_{i_{\newr}}-x_{i_{\newr}^\star}=z_\newr$, for some $z_\newr\in \cY_{\vu,\vec{a}_1}$. Then using this in~\eqref{eq:2nd eq beta_i LRS}, we can write 
\begin{align*}
&a_{2,i_{\newr}} u\(x_{i_{\newr}}\)^2+a_{2,i_{\newr}^\star}u\(x_{i_{\newr}^\star}\)^2\\
&\quad\quad\quad\qquad = \sum_{1\le j_1,j_2\le d} \beta_{j_1}\beta_{j_2}\(a_{2,i_{\newr}}+a_{2,i_{\newr}^\star}(\alpha_{j_1}\alpha_{j_2})^{ z_\newr}\) \(\alpha_{j_1} \alpha_{j_2}\)^{x_{i_\newr}}.
\end{align*}

To show that $\ell \ge k$, it suffices to prove that for any integer $z$, we have 
\[
\(a_{2,i_{\newr}}+a_{2,i_{\newr}^\star}(\alpha_{j_1}\alpha_{j_2})^z\)_{1\le j_1,j_2\le d}\ne 0.\]
This is true for $z=0$, due to the assumptions~\eqref{eqn:a2conds d>2} and~\eqref{eq:sum 0}. Now suppose on the contrary that the above claim fails for some integer $z\ne 0$. Since we assume $d\ge 2$, taking $j_1=j_2=1$ and $2$, we get $\alpha_1^{2z}=\alpha_2^{2z}$, which is a contradiction with the fact that the sequence $\vec{u}$ is non-degenerate. This completes the proof of~\eqref{eq:ell vs k}.

Combining~\eqref{eq:bounds for second eqn} and~\eqref{eq:ell vs k} , we  establish Theorem~\ref{thm:powersumwithLRS} for $a_{1,0}=a_{2,0}=0$.

\subsection{The inhomogeneous case: $\(a_{1,0}, a_{2,0}\)\ne (0,0)$}\label{sec:inhomogeneous LRS}
As in~\eqref{eqn:EXPANDED}, the first equation in~\eqref{eq:geneqnsLRS} can be rewritten as
\[\sum_{\substack{1\le i\le s\\ 1\le j\le d}} \gamma_{i,j}\alpha_j^{x_i}=a_{1,0},\qquad 1\le x_1,\ldots, x_{s} \le H.\]
Moreover, the second equation in~\eqref{eq:geneqnsLRS} can be rewritten as
\[\sum_{1\le i\le s}a_{2,i} \sum_{1\le j_1,j_2\le d} \beta_{j_1} \beta_{j_2} \(\alpha_{j_1} \alpha_{j_2}\)^{x_i}=a_{2,0},\quad 1\le x_1,\ldots, x_{s} \le H.\]

Since $a_{1,0}\ne 0$ or $a_{2,0}\ne 0$, following the proof of~\eqref{eq:sa0}, in both cases we see that there exists a finite subset $\cC_{\vu,\vec{a}_1,\vec{a_2}}\subseteq K^{\ast}$ of cardinality $\sharp\,  \cC_{\vu,\vec{a}_1,\vec{a}_2}\ll 1$, such that 
\[
x_i\in \cC_{\vu,\vec{a}_1,\vec{a}_2},\quad \mathrm{for~some}\quad 1\le i\le s.
\]
 Without loss of generality, we can assume that this is true for $i=s$.

In particular, fixing $x_s$ in $O(1)$ ways,  the system of equations in~\eqref{eq:geneqnsLRS} splits 
into $O(1)$ systems of equations of the form 
\begin{align*}
&a_{1,1}u(x_1)+\ldots+a_{1,s-1}u(x_{s-1}) = \widetilde a_{1,0},\\
&a_{2,1}u(x_1)^2+\ldots+a_{2,s-1}u(x_{s-1})^2 = \widetilde a_{2,0},\\
& \qquad  \qquad 1\le x_1,\ldots, x_{s-1} \le H,
\end{align*}
for some fixed $\widetilde a_{1,0},a_{2,0}\in K^*$.
The proof finally concludes by induction on $s$.

\section{Concluding remarks}

\subsection{Tightness of our results and possible improvements}\label{sec:low}

We now give an example,  which shows that the bound of Theorem~\ref{thm:SystSubgrGen} is tight.

Assume that $(k+1) \mid s$. Then the bound in Corollary~\ref{cor:PowerSumVals} is tight, at least when 
 \begin{equation}\label{eq:allzero}
     a_{1}=\ldots=a_{k}=0
 \end{equation} and $\Gamma$ contains the primitive $(k+1)$th roots of unity. 
  For example, choose  an arbitrary subset $\cV \subseteq \Gamma$ of cardinality $V>0$,
and consider  
\begin{equation}\label{eq: Example W}
\cW =  \left\{ \zeta_{k+1}^e v:~0\le e\le  k,\ v\in \cV \right\}. 
\end{equation}
Then we have $\cW\subseteq\Gamma$ and $V\le \sharp\, \cW=W\le (k+1)V/2$. Moreover, for each $v\in \cV$, we consider the following tuples 
\begin{equation}\label{eq:TV}
\cT(v)=\(v,\zeta_{k+1} v,\ldots,\zeta_{k+1}^{k} v\).
\end{equation}

Now solutions to~\eqref{eq:GenEqnsSubgrAllOne} with~\eqref{eq:allzero} are given by 
\[
(w_1,\ldots,w_s)
=
\(
\mathcal T(v_1),
\ldots,
\mathcal T(v_m)\),
\]
where $v_1,\ldots,v_m\in \cV$. Thus, we obtain at least $V^m=V^{\lfloor s/(k+1)\rfloor} \gg W^{\lfloor s/(k+1)\rfloor}$ solutions, as desired.

A similar construction, with a small adjustment, works also for $s$ with $(k+1)\nmid s$. It may also be used to  show the tightness of Theorem~\ref{thm:VMVT}.

More precisely, for the same set $\cW$ as in~\eqref{eq: Example W} and $\cT$ as in~\eqref{eq:TV}, a set of solution to~\eqref{eq:VMVT} with~\eqref{eq:notpermutation}, is given by
\[
\(\mathcal T(v_1), w_{k+2},\ldots, w_s, \mathcal T(v_2), w_{k+2},\ldots, w_s\),
\]
where $w_{k+2},\ldots, w_s\in \cW$ and $v_1,v_2\in \cV$ are such that $(v_1/v_2)^{k+1} \ne 1$. Since 
\[\sharp\,\left\{(v_1,v_2)\in \cV\times \cV: (v_1/v_2)^{k+1}=1\right\}\le (k+1)V,\]
we obtain at least 
$W^{s-(k+1)}V^2 \gg W^{s-k+1}$ off-diagonal solutions.

While, as we have shown, the bound of Theorem~\ref{thm:SystSubgrGen} is essentially optimal, 
in some special cases we can do better. 
For example, note that if $(w_1,\ldots,w_s)$ is a solution to~\eqref{eq:GenEqnsSubgr}, then $(ww_1,\ldots, ww_s,w)$ is also a solution to the system of equations
\begin{equation}\label{eq:extended system}
\begin{split}
&a_{j,1}u_1^j +\ldots+a_{j,s}u_s^j-a_{j,0}u_{s+1}^j=0, \qquad j =1, \ldots, k,\\
& \qquad  \qquad  \qquad u_1,\ldots, u_{s+1}  \in \cU, 
\end{split}
\end{equation}
where $\cU=\cW^2 =\{w_1w_2:~w_1w_2\in \cW\}$ is the product set of $\cW$.

Now suppose $\cW$ is a set with a small doubling. In a more quantitative form, assume that there exists some $\eta\ge 0$ such that 
\[
\sharp\, \cU=U \ll W^{1+\eta}.
\]
A natural example of such sets (with $\eta = 0$), is given by 
\[
\cW = \{\gamma_1^{x_1}\ldots \gamma_r^{x_r}:~1 \le x_i \le H_i, \ i =1, \ldots, r\}
\]
where $\gamma_1,\ldots, \gamma_r$ are the generators of $\Gamma$.

Moreover, let us assume that the coefficient matrix corresponding to the system~\eqref{eq:extended system} is also non-degenerate as defined in Section~\ref{sec:main}. In particular, we have to assume that $a_{1,0}\in K^{\ast}$. Then applying Theorem~\ref{thm:SystSubgrGen} to this new system of equations, we get
\[M_{\vec{A}}(\cW) \ll  W^{{\lfloor (s+1)/(k+1)\rfloor}(1+\eta)-1}.\]

Let $m$ be be the same quantity as in Theorem~\ref{thm:SystSubgrGen}. In particular, if 
\[
\left\lfloor (s-m)/(k+1)\right\rfloor = \left\lfloor (s+1)/(k+1)\right\rfloor,
\]
we get a better bound than in Theorem~\ref{thm:SystSubgrGen}, over the range
\begin{equation}\label{eq:eta}
0 \le \eta < \frac{1}{\left\lfloor (s+1)/(k+1)\right\rfloor}.
\end{equation}

For example, note that if we have
\[
a_{j,1}=\ldots, a_{j,s}=1 \mand a_{j,0}=-1,\quad j=1,\ldots, k,
\]
then the coefficient matrix associated to~\eqref{eq:extended system} is clearly non-degenerate. 
In this case we have $m=1$ and the above argument applies with $\eta$ 
as in~\eqref{eq:eta} if $s \not \equiv 0, -1 \pmod{k+1}$.

\subsection{Further applications and open questions}\label{sec:comments}
While studying $L_{\va}(H)$ is of independent interest, using the strategy of~\cite{ManOstShp,MOS} and the bounds on $L_{\va}(H)$, one can study the statistics of ranks, determinants and characteristic polynomials  in the families 
of $m \times n$-matrices 
\[
\(u(x_{i,j})\)_{\substack{1 \le i \le m\\ 1 \le j\le n}}, \qquad x_{ij} \in  [1,H].
\]
In particular, using  Theorems~\ref{thm:lrslineqn} and~\ref{thm:powersumwithLRS},  one can obtain analogues of some results from~\cite{BlLi,ManOstShp,MOS} for the above matrices.

Let $k\ge 1$ and $1\le d_1<d_2<\ldots<d_k$ be integers. Motivated by work of   
Salberger and  Wooley~\cite[Theorem~5.2]{SalWool} (see also~\cite[Section~1]{Wool1}), we consider the higher degree system:
\begin{equation*}
\begin{split}
&w_1^{d_j} +\ldots+w_s^{d_j} = w_{s+1}^{d_j} +\ldots+w_{2s}^{d_j},  \qquad j =1, \ldots k,\\
& \qquad  \qquad  w_1,\ldots, w_s, w_{s+1},\ldots, w_{2s}  \in \cW. 
\end{split}
\end{equation*}

Certainly, our methods are capable of handling the case $k=2$. This is because here the analogue of~\eqref{eq:sum wj two sides} is
\[
\sum_{i\in \cI_1}  w_i^{d_j}=\sum_{i\in \cI_2}  w_i^{d_j},\qquad j=1,2,
\]
where $\cI_1\cup \cI_2=\{1,2\}$. To contradict~\eqref{eq:nothingcommon}, it is enough to assume that $d_1,d_2$ are coprime and are of opposite parity. We leave the case $k\ge 3$ as an open question.

\appendix
\section{Counting matrices with a given characteristic polynomial}\label{app:CountMatr_CharPoly}
Denote $\cM_{n}(\cW)$ to be the set of $n\times n$ matrices with entries from $\cW$. Given a polynomial $f\in K[T]$, define
\[
\cP_n(\cW;f)
=\left\{
X\in \cM_n(\cW):
\det(TI_n-X)=f
\right\}.
\]  

Let 
 \begin{align*}
    \alpha(n) = n(n - 1)/2
      &+ \max\Bigl\{
        \fl{(n - 1)/2} + \fl{n(n - 1)/4},\\
      & \qquad \qquad  \qquad \qquad\fl{n/2}  + \fl{(n(n - 1)/4-1/4}
      \Bigr\}
 \end{align*}
be as in~\cite[Equation~2-2]{ManOstShp}. Inserting Corollary~\ref{cor:PowerSumVals} 
(for $k=2$) in the argument of~\cite{ManOstShp}, we improve the bound $\sharp\, \cP_n(\cW;f)\ll W^{\alpha(n)}$ of~\cite[Theorem~2.4]{ManOstShp} for $n\equiv1\pmod4$.  

\begin{thm}\label{thm:givencharpoly}
For any monic polynomial $f$ of degree $n\ge 3$, and any
$\cW \subseteq \Gamma$ of finite cardinality $W$, we have
\[
\sharp\, \cP_n(\cW;f)\ll W^{\alpha_{0}(n)},
\]
where
\[\alpha_{0}(n)
=
\begin{cases}
\alpha(n)-1, & n\equiv1\pmod4,\\
\alpha(n), & n\not\equiv1\pmod4.
\end{cases}\]
\end{thm}

We also obtain more precise bounds depending on the second coefficient of $f$.

\begin{thm}\label{thm:givencharpolyprecise} 
For $f = T^n + c_{n-1}T^{n-1} + \ldots + c_0 \in K[T]$ of degree $n\ge 3$, and any
$\cW \subseteq \Gamma$ of finite cardinality $W$, we have
\[
\sharp\, \cP_n(\cW;f)
\ll
W^{\beta(n)}
\begin{cases}
W^{-\lambda_{0}(n)}, & c_{n-1}=0,\\
W^{-\nu_{0}(n)}, & c_{n-1}\ne0,
\end{cases}
\] 
where $\beta(n)=(3n^2-n)/4$,
\[
\lambda_{0}(n)=
\begin{cases}
1, & n\equiv 0,3 \pmod4,\\
3/2, & n\equiv 1 \pmod4,\\
1/2, & n\equiv 2 \pmod4,
\end{cases}
\]
and
\[
\nu_{0}(n)=
\begin{cases}
2, & n\equiv 0 \pmod 4,\ n\ne4,\\
3/2, & n\equiv 1,2 \pmod 4,\\
1, & n\equiv 3 \pmod 4\text{ or }n=4.
\end{cases}
\]
\end{thm}

Our exponents are all \underline{one} less than the corresponding exponents 
in~\cite[Theorem~B.1]{ManOstShp} in the following cases.
\[
\begin{array}{|c|c|}
\hline
\text{condition on $f$} & \text{condition on $n$}\\
\hline
c_{n-1}=c_{n-2}=0 & n=5 \\
c_{n-1}=0,\ c_{n-2}\ne 0 & n \equiv 1\pmod 4 \\
c_{n-1}\ne 0 & n \equiv 0,1\pmod 4,\ n \ne 4 \\
\hline
\end{array}
\]

Moreover, in the other remaining cases, our exponents are the same  
as in~\cite[Theorem~B.1]{ManOstShp}.

\section*{Acknowledgement}
During the preparation of this work, the authors were partially supported by the Australian Research Council Grants DP230100530 and  FT250100208. I.S. also would like to thank the Max Planck Institute for Mathematics, Bonn, for its support and excellent working conditions during the preparation of this work.


\begin{thebibliography}{9}
\bibitem{AV09} F. Amoroso and E. Viada, 
`Small points on subvarieties of a torus', 
{\it Duke Math. J.\/} {\bf 150} (2009), 407--442.

\bibitem{AV11} F. Amoroso and E. Viada, 
`On the zeros of linear recurrence sequences', 
{\it Acta Arith.\/} {\bf 147} (2011), 387--396.


\bibitem{BBG}
J. Bajpai, S. Bhakta and V.~C. Garc\'ia, `Exponential sums in prime fields for modular forms', 
{\it Res. Number Theory\/}, {\bf 8} (2022),  Art.18. 


\bibitem{BlLi} V. Blomer and  J. Li, `Correlations of values of random diagonal forms', 
{\it Intern. Math. Res.  Notices\/}, {\bf 2023} (2023),  20296--20336. 


\bibitem{BGKS}
J. Bourgain, M. Z. Garaev, S. V. Konyagin, and I. E. Shparlinski,
`Multiplicative congruences with variables from short intervals',
{\it J. Anal. Math.\/}, {\bf 124} (2014), 117--147.


\bibitem{Lang} S. Lang, 
{\it Fundamentals of Diophantine geometry\/}, 
Springer-Verlag, New York, 1983.

 \bibitem{ManOstShp} A. Manning, A. Ostafe  and I. E. Shparlinski,
 `Counting matrices over finite rank multiplicative groups', {\it  Pacific J. Math.\/}, {\bf 340}  (2026), 115--138. 

\bibitem{MOS} A. Mohammadi, A. Ostafe and I. E. Shparlinski,  `On some matrix counting problems',
{\it J. Lond. Math. Soc.\/}, {\bf 110} 2024, Art.~e70044.

\bibitem{Mord}
L. J. Mordell,  `On a sum analogous to a Gauss sum', 
{\it  Quart. J. Math.\/}, {\bf 3} (1932),  161--167.

\bibitem{Mud} A. Mudgal,  `Diameter-free estimates for the quadratic Vinogradov mean value theorem',
{\it Proc. Lond. Math. Soc.\/}, {\bf 126} (2023),   76--128. 

\bibitem{SalWool} P. Salberger and T. D. Wooley, `Rational points on complete intersections of higher degree, and mean values of Weyl sums',
{\it J. Lond. Math. Soc.\/}, {\bf 82} (2010), 317--342.


\bibitem{Wool1}  T. D. Wooley, `Paucity problems and some relatives of Vinogradov’s
mean value theorem',  {\it Math. Proc. Camb. Phil. Soc.\/}, {\bf 175} (2023), 327--343. 

\bibitem{Wool2}  T. D. Wooley, `Condensation and densification for sets of large diameter', 
{\it Combinatorial and Additive Number Theory VI\/}, 
Springer Proc. Math. Stat., v.464, Springer, Cham, 2025, 461--502.

\bibitem{Wool3}  T. D. Wooley, `Strong paucity in the Br{\"u}dern--Robert Diophantine system',  
{\it Quart. J. Math.\/}, (to appear). 

\end{thebibliography}
\end{document}